\documentclass[12pt,a4paper]{amsart}
 \pdfoutput=1
 
 \usepackage{amsmath,amssymb,amsthm}  
 \usepackage{mathtools}
 \usepackage{tikz}
 \usepackage{thmtools}
 \usepackage{float} 
 
 \usepackage{xcolor} 	
 \usepackage{hyperref}
 \hypersetup{
 	colorlinks,
     linkcolor={red!60!black},
     citecolor={green!60!black},
     urlcolor={blue!60!black},
 }

 \usepackage[maxbibnames=99]{biblatex}
 \usepackage[utf8]{inputenc}
 \usepackage[T1]{fontenc}
 \usepackage{lmodern}
 \usepackage[babel]{microtype}
 \usepackage[english]{babel}
 
 \usepackage{geometry}
 \usepackage{enumitem}

\theoremstyle{plain}
\newtheorem{thm}{Theorem}[section]

\newtheorem{prop}[thm]{Proposition}

\newtheorem{obs}[thm]{Observation}

\newtheorem{conj}[thm]{Conjecture}

\newtheorem*{thm*}{Theorem}
\newtheorem*{cor*}{Corollary}

\theoremstyle{definition}

\title{A Cantor-Bernstein theorem for infinite matroids}

\author{Attila Jo\'{o}}
\thanks{The author would like to thank the generous support of the Alexander 
von Humboldt Foundation and NKFIH 
OTKA-129211}
\address{Attila Jo\'{o},
Department of Mathematics, University of Hamburg, Bundesstra{\ss}e 55 (Geomatikum), 20146 Hamburg, Germany}
\email{attila.joo@uni-hamburg.de}
\address{Attila Jo\'{o},
Logic, Set theory and topology department, Alfr\'{e}d R\'{e}nyi Institute of Mathematics,  13-15 Re\'{a}ltanoda St., 
Budapest, Hungary}
\email{jooattila@renyi.hu}


\begin{document}

\begin{abstract}
We give a common matroidal generalisation of  `A Cantor-Bernstein theorem for paths in graphs' by Diestel and 
Thomassen and `A Cantor-Bernstein-type theorem for spanning trees in infinite graphs' by ourselves. 
\end{abstract}
\maketitle

\section{Introduction}

Let us reformulate the Cantor-Bernstein theorem in the language of graph theory:
\begin{thm}[Cantor-Bernstein,  \cite{cantor1987}]
If  $ G=(V_0,V_1;E) $ is a  bipartite graph and matching $ I_i $ covers $ V_i $ for $ i\in \{ 0,1 \} $, then $ G $ admits a perfect 
matching.
\end{thm}

Ore discovered the following generalisation of the Cantor-Bernstein theorem which is the 
extension of the Mendelsohn-Dulmage theorem \cite[Theorem 1]{mendelsohn1958some} to infinite graphs:
\begin{thm}[Ore, {\cite[Theorem 7.4.1]{ore1962thetheory}}]\label{Ore}
Let $ G=(V_0,V_1;E) $ be a  bipartite graph and let $ I_0, I_1\subseteq E $ be matchings in $ G $. Then there exists 
a matching $ I $  such that $ V(I)\cap V_i\supseteq V(I_i)\cap V_i $ for $ i\in \{ 0,1 \} $.
\end{thm}
Diestel and Thomassen  examined in their paper `A Cantor-Bernstein theorem for paths in graphs' a more general 
graph-theoretic setting in which disjoint paths are used to connect two vertex sets. We call a finite path that meets the vertex 
sets $ V_0 $ and  $ V_1 
$ and subgraph-minimal with respect to this property a $ V_0V_1 $-path.
\begin{restatable}[Diestel and Thomassen, \cite{diestel2006cantor}]{thm}{Pym}\label{Pym}
Assume that $ G=(V,E) $ is a graph, $ V_0, V_1\subseteq V$   and  $ \mathcal{P}_i $ is a system of disjoint $V_0V_1 
$-paths in 
$ G $ for $  i\in \{ 0,1 \} $. Then there exists a system of disjoint $ V_0V_1 $-paths $ \mathcal{P} $ with $ V(\mathcal{P})\cap 
V_i \supseteq V(\mathcal{P}_i)\cap V_i $ for $ i\in \{ 0,1 \} $.  
\end{restatable}
Note that Theorem \ref{Ore} is the special case of Theorem \ref{Pym} where $ G $ is bipartite and the sets $ V_i $ are its vertex 
classes.  

 In our paper entitled `A Cantor-Bernstein-type theorem for spanning trees in infinite graphs' we investigated  if the existence of a 
 $ \kappa $-packing and a $ \kappa $-covering by spanning trees 
implies the existence of a $ \kappa $-family of spanning trees which is both, i.e. a $ \kappa $-partition: 
\begin{thm}[Erde et al. {\cite[Theorem 1.1]{erde2021cantor}}]\label{tree partition}
Let $ G=(V,E) $ be a  graph and let  $ \kappa $ be a cardinal. If there are $ \kappa $ many pairwise edge-disjoint 
spanning trees in $ G $ and $ E $ can be covered by $ \kappa $ many spanning trees, then $ E $ can be partitioned into $ \kappa $ 
many spanning trees.
\end{thm}
At first sight the connection between Theorem \ref{Pym} and Theorem \ref{tree partition} 
seems to be only analogical. In this paper, we show that the connection is actually stronger. There is an abstract matroidal 
``Cantor-Bernstein''-type phenomenon behind these theorems.  Let us first state  a special case of our main result which is 
the generalisation of a theorem by 
Kundu and  Lawler 
(see \cite{kundu1973matroid}) to finitary matroids\footnote{A matroid is called finitary if all of its circuits 
are finite. In the older papers of Higgs, Oxley and others it is also called `independence space'. For a brief introduction to the concept of infinite 
matroids see Section \ref{sec: Prelim}.}:
\begin{restatable}{thm}{infmatOrew}\label{thm: infmatOrew} 
For $ i\in \{ 0,1 \} $, let $ M_i  $ be a finitary matroid on $ E $  and let  $I_i\in 
\mathcal{I}_{M_{0}}\cap \mathcal{I}_{M_{1}}  $. Then there is an $ I\in  \mathcal{I}_{M_0}\cap \mathcal{I}_{M_1} $ 
with $I_i\subseteq 
 \mathsf{span}_{M_{i}}(I) $ for $ 
 i\in \{ 0,1 \} $.
\end{restatable}

The proof for finite matroids by Kundu and Lawler  in \cite{kundu1973matroid} is quite short: If $ I_0 $ spans $ I_1 $ in $ 
M_1 $, then $ I:=I_0 $ is as desired. Otherwise we add an $ e\in I_1 \setminus \mathsf{span}_{M_1}(I_0) $ to $ I_0 $ and if $ I_0+e\notin 
\mathcal{I}_{M_0} 
$, then
delete a suitable $ f\in I_0\setminus I_1 $ in order to restore the $ M_0 $-independence. This 
can be done because the fundamental circuit $ C_{M_0}(e, I_0) $ (if exists) cannot be entirely in $ I_1 $. The resulting set $ 
I_0+e-f $ 
(or $ I_0+e $) still spans  $ I_0 $ in $ M_0 $ and has strictly more edges in $ I_1 $ than $ I_0 $. After finitely many iterations 
of this step the desired $ I $ is obtained.

A naive proof-idea for Theorem \ref{thm: infmatOrew} would be to iterate the step above via transfinite recursion. 
Unfortunately it does not work. To demonstrate this we define  a graph $ G=(V,E) $ as a ray (one-way infinite path) $ v_0, v_1, v_2,\dots $ 
together with an additional vertex $ w $ connected to each vertex of the ray (see Figure \ref{fig: CBM greedy fails}). Let $ M_0 $ be the  cycle 
matroid 
on $ E $ 
corresponding to $ G $ (i.e. the circuits are the edge sets of the graph-theoretic cycles) and let $ M_1 $ be the free matroid on 
$ E $ (i.e. every set is independent in $ M_1 $). We define $ I_0 $ as the set of edges incident with $ w $ and let $ I_1:=E \setminus I_0 $. The 
naive approach might 
proceed as: \[ I_0,\ I_0+v_0v_1-wv_0,\  I_0\cup \{ v_0v_1, v_1v_2 \} \setminus \{ wv_0, w_1 \}, \dots  \] 

\begin{figure}[h]
\centering
\begin{tikzpicture}[scale=2]
\node[inner sep=0pt] (v2) at (-1,-2) {$v_0$};
\node[inner sep=0pt] (v3) at (0,-2) {$v_1$};
\node[inner sep=0pt] (v4) at (1,-2) {$v_2$};
\node[inner sep=0pt] (v5) at (2,-2) {$v_3$};
\node[inner sep=0pt] (v6) at (3,-2) {$v_4$};
\node[inner sep=0pt] at (3.5,-2) {$\dots$};
\node[inner sep=0pt] (v1) at (-1,-3) {$w$};

\draw (v1) edge (v2);
\draw  (v1) edge (v3);
\draw  (v1) edge (v4);
\draw  (v1) edge (v5);
\draw  (v1) edge (v6);

\draw[dashed]  (v2) edge (v3);
\draw[dashed]  (v3) edge (v4);
\draw[dashed]  (v4) edge (v5);
\draw[dashed]  (v5) edge (v6);
\node[dashed] at (3.5,-2.5) {$\dots$};
\end{tikzpicture}
\caption{The failure of the naive approach for infinite matroids.} 
\label{fig: CBM greedy fails}
\end{figure}
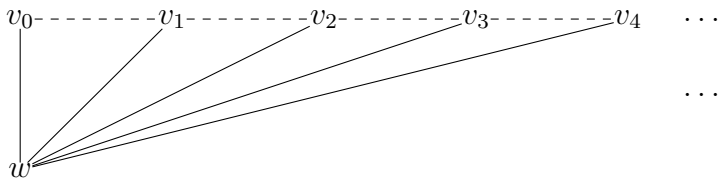

\noindent It terminates after $ \omega $ steps and transforms $ I_0 $ into $ I_1 $. Since $ I_1 $ does not span $ I_0 $ in $ 
M_0 $, it fails to provide a desired $ I $. It is easy to see that if we keep $ wv_0 $ and delete only $ wv_1, wv_2, \dots $ (while the incoming edges 
are 
in the same order), then we end up with the same ray together with the edge $ wv_0 $ which is suitable as $ I $. In order to 
prove  
Theorem \ref{thm: infmatOrew}, we are going to 
show in Section \ref{sec: infinit KuLa} that it is always possible to choose the leaving edge in each step in such a way that 
we obtain a solution at the end.  The proof 
of Theorem \ref{thm: infmatOrew}  makes it possible to understand 
quickly the 
main ideas without dealing with technicalities arising in the general form. Basic knowledge about finite matroids is already sufficient to 
understand the paper, all the necessary matroidal background is given in Section \ref{sec: Prelim}.

 In Section \ref{sec: proofmain} we discuss the general form of our main result. Let us denote the class of 
finitary matroids  by $ \mathfrak{F} $, the class of their duals (i.e. cofinitary matroids) by $ \mathfrak{F}^{*} $ and let $ \mathfrak{F}\oplus 
\mathfrak{F}^{*} $ be the class of matroids  that are the 
direct sums of a finitary and a cofinitary matroid (equivalently the matroids with only finitary and cofinitary components).  For a matroid class $ 
\mathfrak{C} $, let $ \mathfrak{C}(E) $ be the set of matroids on edge 
set $ E $ that are in 
class 
$ \mathfrak{C} $. 

\noindent  Our main result generalises Theorem \ref{thm: infmatOrew} in  two ways. On the one 
hand, we replace $ \mathfrak{F} $ by $ \mathfrak{F}\oplus \mathfrak{F}^{*} $. On the other hand, we allow arbitrary edge sets 
instead of  common independent sets (this possibility was conjectured by Bowler) in the following sense:

\begin{restatable}{thm}{infmatOre}\label{thm: infmatOre} 
  For $ i\in \{ 0,1 \} $, let $ M_i\in (\mathfrak{F}\oplus \mathfrak{F}^{*})(E)  $  and  $ F_i\subseteq E $. Then  there exists an  $ 
  F\subseteq E $ such that
  $ \mathsf{span}_{M_i}(F)\supseteq F_i $ and $ \mathsf{span}_{M_i^{*}}(E\setminus F)\supseteq E\setminus F_{1-i} $ for 
  $ i\in \{ 0,1 \} $.
\end{restatable}

We are going to prove the following family variant of Theorem \ref{thm: infmatOre} as well: 

\begin{restatable}{thm}{main}\label{thm: family CB}
  For $  i\in \Theta$, let $ M_i\in (\mathfrak{F}\oplus \mathfrak{F}^{*})(E) $,  $ P_i,R_i \subseteq E $ and for $ e\in E $, let $ N_e\in 
  (\mathfrak{F}\oplus \mathfrak{F}^{*})(\Theta) $.  Then there are $ T_i\subseteq P_i\cup R_i $ for $ i\in 
  \Theta $  such that 
  \begin{enumerate}
  \item $ \mathsf{span}_{M_i}(T_i)\supseteq P_i $;
  \item $\mathsf{span}_{M_i^{*}}(E\setminus T_i)\supseteq E\setminus R_i $;
  \item For every $ e\in E $, the set $ \{  i\in \Theta:\ e\in T_i  \} $ spans $ \{  i\in \Theta:\ e\in R_i  \} $ in $ N_e $;
  \item For every $ e\in E $, the set $  \{  i\in \Theta:\ e\notin T_i  \} $ spans $ \{  i\in \Theta:\ e\notin P_i  \} $ in $ 
  N^{*}_e $.
  \end{enumerate}
\end{restatable}

\noindent The connection between the Theorems \ref{thm: infmatOre} and \ref{thm: family CB} is far from obvious.  It worths to mention 
that it is impossible to extend our results 
above to arbitrary matroids working in set theory ZFC. Indeed, the 
analogue of Theorem  
\ref{thm: infmatOrew} for arbitrary matroids fails under the Continuum Hypothesis even if $ E $ is countable,  $ M_i $ 
is uniform and $ I_i $ is a base of $ M_i $ (take $ U $ and $ U^{*} $ in  
\cite[Theorem 5.1]{erde2019base}).

In the last section (Section \ref{sec: Appl}) we provide an application  related to the following conjecture:

\begin{conj}[Matroid Intersection Conjecture by Nash-Williams, {\cite[Conjecture 
1.2]{aharoni1998intersection}}]\label{MIC}
For every  $ M_0, M_1\in \mathfrak{F}(E) $, there is an  
  $ I \in \mathcal{I}_{M_0}\cap \mathcal{I}_{M_1}$ and a partition $ E=E_0\sqcup E_1 $ such that $ I\cap E_i $ spans $ 
  E_i $ in $ 
 M_i $ for $ i\in \{ 0,1 \} $.  
\end{conj}

\noindent The special case of the conjecture where $ E $ is assumed to 
be countable was proved in \cite{joo2020MIC}. This was then generalised 
 to the case where $ E $ is still countable but $ \mathfrak{F}(E) $ is replaced by $ (\mathfrak{F}\oplus 
\mathfrak{F}^{*})(E) $ (see \cite[Theorem 1.4]{joo2021packingcovering}).  

A maximal sized common independent set of two finite matroids can always be chosen in such a way that it spans a prescribed common 
independent set in both  matroids.  Indeed, if a common independent set is not a largest such a set, then the well-known  `augmenting path' method 
by Edmonds gives a new common independent set which is larger by one and spans the original in both matroids  (see in
\cite{edmonds2003submodular}). Iterating such augmenting paths starting with the prescribed common independent set provides a desired largest 
common independent set.

  The question can be phrased with respect to Conjecture \ref{MIC} by replacing `maximal sized' by `strongly maximal' which we define as 
  satisfying the property described in Conjecture \ref{MIC}. The same argument for the positive answer does not work because finitely many 
  iteration of augmenting paths does not lead to a strongly maximal one in general. Even so, we can answer the question affirmatively based on our 
  main results. Let us denote the set of strongly maximal common independent sets  by $ 
 \mathsf{SM}(M_0, M_1) $.   For $ I, J\in  \mathcal{I}_{M_0}\cap \mathcal{I}_{M_1}$, let  $ J\trianglelefteq_{M_0, M_1} 
 I $ iff $ J\subseteq \mathsf{span}_{M_0}(I)\cap 
 \mathsf{span}_{M_1}(I) $.

\begin{restatable}{thm}{cofinal}\label{MIC extend}
  Let $ E $ be countable and let $ M_i\in (\mathfrak{F}\oplus\mathfrak{F}^{*})(E) $  for $ i\in \{ 0,1 \} $.  Then $ \mathsf{SM}(M_0,M_1) $  is 
  cofinal but not necessarily upward closed in $ (\mathcal{I}_{M_0}\cap \mathcal{I}_{M_1}, \trianglelefteq_{M_0, M_1}) $.  
\end{restatable}

\section{Preliminaries}\label{sec: Prelim}
Rado asked in 1966 if there is an infinite generalisation of matroids  preserving the key concepts 
(like duality and minors) of 
the finite theory. Based on some early results of 
Higgs \cite{higgs1969matroids} and Oxley  \cite{oxley1992infinite},   Bruhn, Diestel, Kriesell, Pendavingh and Wollan  answered 
the 
question 
affirmatively and gave a set of cryptomorphic axioms for infinite matroids, generalising the usual independent set-, bases-, 
circuit-, closure- and 
rank-axioms of finite matroids (see \cite{bruhn2013axioms}). They showed that several fundamental  facts of the theory of finite 
matroids are preserved in the infinite case. It opened the door  
for a more systematic investigation of infinite matroids.
An $ M=(E, \mathcal{I}) $ is a matroid (also called B-matroid)  if $ \mathcal{I}\subseteq 
\mathcal{P}(E) $ with
\begin{enumerate}
[label=(\Roman*)]
\item\label{item axiom1} $ \varnothing\in  \mathcal{I} $;
\item\label{item axiom2} $ \mathcal{I} $ is downward closed;
\item\label{item axiom3} For every $ I,J\in \mathcal{I} $ where  $J $ is $ \subseteq $-maximal in $ \mathcal{I} $ and $ I $ is 
not, there 
exists an
 $  e\in 
J\setminus I $ such that
$ I+e\in \mathcal{I} $;
\item\label{item axiom4} For every $ X\subseteq E $, any $ I\in \mathcal{I}\cap 
\mathcal{P}(X)  $ can be extended to a $ \subseteq $-maximal element of 
$ \mathcal{I}\cap \mathcal{P}(X) $.
\end{enumerate}
 For a finite~$E$, axioms  \ref{item axiom1}-\ref{item axiom3}  are equivalent to the usual axiomatization of finite
matroids in terms of 
independent sets  (while \ref{item axiom4} is automatically true). 

The terminology and the basic facts we will use are  well-known for 
finite matroids. The elements of~$\mathcal{I}$ are called \emph{independent} sets  while the sets 
in $\mathcal{P}(E) \setminus \mathcal{I}$ are 
\emph{dependent}. The maximal independent sets are the \emph{bases} and  the minimal dependent sets are the 
\emph{circuits} of the matroid. Every dependent set contains a circuit (which fact is not obvious if $ E $ is infinite). A singleton 
circuit is called 
a \emph{loop}. The \emph{components} of a matroid are the connected components of the hypergraph of its circuits on $ E $.     
The \emph{dual} of  
matroid~${M}$   is the 
matroid~${M^*}$  on the same edge set whose bases are the complements of 
the bases of~$M$. By the deletion of an $ X\subseteq E $  we obtain the matroid
$ M-X:=(E\setminus X, \{ Y\in \mathcal{I}:  Y\subseteq E\setminus X \}) $ and the contraction of $ X $ gives $M/X:= 
(M^{*}-X)^{*} $. If $ I $ is independent in $ M $  but $ I+e $ is dependent for some $ e\in E\setminus I $  then there is a unique 
circuit   $ C_M(e,I) $ of $ M $ through $ e $ contained in $ I+e $ which is called the \emph{fundamental circuit} of $ e $ on $ I $ in $ M $.   We 
say~${X 
\subseteq E}$ spans~${e \in E}$ in matroid~$M$ if either~${e \in X}$ or there exists a circuit~${C 
\ni e}$ with~${C-e \subseteq X}$. 
We denote the set of edges spanned by~$X$ in~$M$ by~$\mathsf{span}_{M}(X)$. 
A matroid is called \emph{finitary} if all of its circuits are finite. A matroid is 
\emph{cofinitary} if its dual is finitary. 
If $ C_1$ and $C_2 $ are 
circuits with $ e\in C_1\setminus C_2 $ and $f\in  C_1\cap C_2 $, then  there is a circuit $ C_3 $ with $e\in C_3\subseteq 
C_1\cup C_2-f $. This fact is called (strong) \emph{circuit elimination}. For more information about 
infinite matroids we refer to  \cite{nathanhabil}.

\section{The infinite generalisation of the  Kundu-Lawler theorem}\label{sec: infinit KuLa}

\infmatOrew*

\begin{proof}
We may assume without loss of generality that $ E $ is the disjoint union of $ I_0 $ and $  I_1 $ 
since otherwise we can simply contract $ I_0 \cap I_1 $ and delete $ E\setminus (I_0 \cap I_1) $ in both matroids. Let 
$ < $ be a well-order on $ E $ in which $ I_1 $ is an initial segment, i.e. $ e< f$ for every $ e\in I_1 $ and $ f\in 
I_0 $. From now on, the 
maximum of a finite subset of $ E $ is interpreted corresponding to $ <$. We define a well-order 
$ \prec $ on the set $ E^{<\aleph_0} $ of finite subsets of $ E  $. For $ X\neq Y\in  E^{<\aleph_0} $ 
let $ X\prec Y $ iff  one of the following holds:
\begin{itemize}
\item $ X=\varnothing $,
\item $ \max X<\max Y $,
\item $ \max X=\max Y =:z$ and $ X-z\prec Y-z $.
\end{itemize}
It is not too hard to check that $ \prec $ is indeed a well-order.

\begin{obs}\label{obs:common part}
If $ X\prec Y $ then $  X+z\prec Y+z  $ for every $ z\in I_0\cup I_1 $.
\end{obs}

  Let $ \left\langle E_\beta: \beta<\alpha \right\rangle  $ be a sequence of subsets of $ E $ where $ 
\alpha $ is a limit ordinal. If

\[ \bigcup_{\gamma<\alpha}\bigcap_{\beta>\gamma}E_\beta= 
\bigcap_{\gamma<\alpha}\bigcup_{\beta>\gamma}E_\beta,\]
then we call this set the limit of the sequence and denote it by $ \lim\left\langle E_\beta: \beta<\alpha \right\rangle  $. We apply transfinite 
recursion starting with 
  $ J_0:=I_0 $. Suppose that $ J_\alpha \in \mathcal{I}_{M_0}\cap \mathcal{I}_{M_1} $ is defined and 
spans $ I_0 $ in $ M_0 $. If $ J_\alpha $ spans $ I_1 $ in $ M_1 $ as well, then $ I:=J_\alpha $ is as desired. Otherwise let $ e\in I_1\setminus 
\mathsf{span}_{M_1}(J_\alpha) $ be arbitrary and let

\[ J_{\alpha+1}:=\begin{cases} J_\alpha+e &\mbox{if } \text{ it is independent in }M_0 \\
J_\alpha+e-\max C_{M_0}(e, J_\alpha) & \mbox{ otherwise}.  
\end{cases}    \]

Note that  $ e\in I_1 \setminus I_0 $ and $ \max C_{M_0}(e, J_\alpha) \in I_0 \setminus I_1 $. In  limits steps we take the 
limit of the earlier 
members (which is well-defined). Clearly, $ J_\alpha \in \mathcal{I}_{M_0}\cap \mathcal{I}_{M_1} $ remains true for limit ordinals because a 
finite circuit cannot show up first in a limit step.  It is enough to 
show that $ J_\beta\subseteq  \mathsf{span}_{M_0}(J_\alpha) $ for $ \beta<\alpha $. Let $ \beta $ and  $ g\in I_\beta $ be 
fixed and 
suppose for a contradiction that there is a (smallest) $ \alpha $ with  $ g\notin \mathsf{span}_{M_0}(J_\alpha) $. It is 
obvious from the definition 
of successor steps that $  \alpha $ must be a limit ordinal.  For $\gamma\in [\beta, \alpha) $, let $ S_{\gamma} $ be the 
unique minimal subset of $ 
J_\gamma $ 
that spans $ g $ in $ M_0 $.  It is enough to show that $ 
S_{\gamma+1} \preceq S_\gamma $ for $\gamma\in [\beta, \alpha) $. Indeed, since there is no infinite $ \prec $-decreasing 
sequence, $ S_\gamma 
$ is the same set $ S $ for every large enough $ \gamma $. But then $ S\subseteq J_\alpha $ and it spans $ g $ in $ M_0 $, a 
contradiction. 

Let $\gamma\in [\beta, \alpha) $ be fixed. We may assume that $ S_{\gamma+1}\neq S_\gamma $ since otherwise we are 
done. Suppose first 
that $ S_\gamma=\{ g \} $. Then $ g\notin  S_{\gamma+1} $ because otherwise $ S_\gamma=S_{\gamma+1}=\{ g \} $. 
But then there is an edge $ e $ such that $  g=\max C_{M_0}(e, J_\alpha) $ and $ J_{\gamma+1}=J_\gamma+e-g 
$. Therefore   \[ S_{\gamma+1}=C_{M_0}(e, J_\gamma)-g \prec\{ g \}= S_\gamma. \]  
If $ S_\gamma\neq\{ g \} $, then $ S_\gamma=C_{M_0}(g,J_{\gamma})-g $ and there is an edge $ e $ such that $ 
J_{\gamma+1}=J_\gamma+e-\max C_{M_0}(e, J_\gamma) $ with $ \max C_{M_0}(e, J_\gamma)\in 
C_{M_0}(g,J_{\gamma})-g $. By strong circuit elimination we know that  \[  C_{M_0}(g, J_{\gamma+1})\subseteq 
C_{M_0}(g, J_{\gamma}) \cup C_{M_0}(e, J_\gamma)- \max C_{M_0}(e, J_\gamma) \]
and therefore \[ S_{\gamma+1}\subseteq S_\gamma \cup C_{M_0}(e, J_\gamma) - \max C_{M_0}(e, J_\gamma).\]

It follows that $ S_{\gamma+1}\setminus S_{\gamma}\prec S_\gamma \setminus S_{\gamma+1} $ because $ \max 
C_{M_0}(e, J_\gamma)\in S_{\gamma+1}\setminus S_{\gamma} $ is $ < $-larger than any element of $ S_\gamma 
\setminus S_{\gamma+1} $. Finally, this implies
$ S_{\gamma+1}\prec S_\gamma $ by applying Observation \ref{obs:common part} repeatedly with the edges in $ 
S_{\gamma}\cap S_{\gamma+1} $.
\end{proof}

\section{The proof of the main results}\label{sec: proofmain}
We are going to derive Theorems \ref{thm: infmatOre} and \ref{thm: family CB} from the following statement: 
 
\begin{prop}\label{prop: family CBw}
For $  i\in \Theta$, let $ M_i\in (\mathfrak{F}\oplus \mathfrak{F}^{*})(E) $ and   $ P_i,R_i \subseteq E $  such that the sets $ P_i $ form a 
packing and the sets $ R_i $ form a covering, i.e. $ P_i\cap 
  P_j=\varnothing $ for $ i\neq j$ and $ \bigcup_{i\in \Theta}R_i=E $.  Then there are $ T_i\subseteq P_i\cup R_i $ for $ i\in 
  \Theta $ forming a partition of $ E $ such that $ \mathsf{span}_{M_i}(T_i)\supseteq P_i $ and 
  $\mathsf{span}_{M_i^{*}}(E\setminus T_i)\supseteq E\setminus R_i $.  
\end{prop}

\begin{proof}
We may assume without loss of generality by ``trimming''  that the sets $ R_i $  form a partition of $ E $. We can 
also assume 
that $ P_i\in \mathcal{I}_{M_i} $ since otherwise we replace $ P_i $ with a maximal $ M_i $-independent subset of it. It is 
enough 
to consider the case 
where $ P_i\cap R_i=\varnothing $ for $ i\in \Theta $. Indeed, if it is not the case, then we contract $ P_i\cap R_i $ 
 and delete $ P_j\cap R_j $ for $ j\neq i $ in $ M_i $.  Finally, by decomposing each $ M_i $ into a 
finitary and a cofinitary matroid (which we extend to $ E $ by loops) and partition the sets $ R_i $ and $ P_i $ accordingly,  
it is enough to deal with matroid families where each $ M_i $  is either finitary or cofinitary. 

Let $ <_i $ be a well-order on $ P_i\cup R_i $ where $ R_i$  is an initial segment. Then $ <_i $ induces a well-order $ 
\prec_i $ on the set $ \left[P_i\cup R_i \right]^{<\aleph_0} $ the same way as in Section \ref{sec: infinit KuLa}.

\begin{obs}\label{limit obs}
Suppose that $ E_\alpha $ is the limit of  $ \left\langle E_\beta: \beta<\alpha \right\rangle  $.
\begin{enumerate}
[label=(\roman*)]
\item If $ E_\alpha $ contains an  $ M_i $-circuit $ C\not \subseteq R_i $ where $ M_i $ is finitary, then so does $ E_\beta $ for 
every large enough $ \beta<\alpha $;
\item If $g\in \mathsf{span}_{M_i}(E_\beta) $ for $ \beta<\alpha $ where $ M_i $ is cofinitrary, then $ 
g\in \mathsf{span}_{M_i}(E_\alpha) $.
\end{enumerate}
\end{obs}

To construct the desired partition $ (T_i: i\in \Theta) $, we apply transfinite 
recursion.  Let $ T_i^{0}:=P_i $ for $ i\in \Theta $. Suppose that $ 
T_i^{\beta} $ is defined 
for $ \beta<\alpha $ and $ i\in \Theta $ satisfying the following properties:

\begin{enumerate}
\item\label{TRP 1}$T_i^{\beta}\cap T_j^{\beta}=\varnothing$ for $ i\neq j\in \Theta $;
\item\label{TRP 2} $ T_i^{\beta}\subseteq P_i\cup R_i $;
\item\label{TRP 3} $ T_i^{\beta}\cap P_i $ is $ \subseteq $-decreasing  and $ T_i^{\beta}\cap R_i $ is $ \subseteq 
$-increasing  
in $ 
\beta $;
\item\label{TRP 4} $ T_i^{\beta}=\lim\left\langle T_i^{\delta}: \delta<\beta  \right\rangle $ if $ \beta $ is a limit 
ordinal;
\item\label{TRP 5} $ \mathsf{span}_{M_i}(T_i^{\beta}) \supseteq P_i$;
\item\label{TRP 6} For every finitary $ M_i $, each $ M_i $-circuit $ C\subseteq T_i^{\beta} $ is a subset of $ R_i $;
\item\label{TRP 7} For every finitary $ M_i $ and $ g\in P_i $, the $ \prec_i $-smallest finite $ S^{\beta}_g \subseteq 
T_i^{\beta}$ that is 
witnessing 
$g\in  \mathsf{span}_{M_i}(T_i^{\beta}) $ is a $ \preceq_i $-decreasing function of $ \beta $;
\item\label{TRP 8} $ (T_i^{\delta}: i\in \Theta)\neq (T_i^{\delta+1}: i\in \Theta) $ for $ \delta+1<\alpha $.
\end{enumerate}

Note that condition (\ref{TRP 6}) is  a rephrasing of ``$ \mathsf{span}_{M_i^{*}}(E\setminus T_i^{\beta})\supseteq 
E\setminus R_i $ for finitary $ M_i $''. Assume first that $ \alpha $ is a limit ordinal. Then conditions (\ref{TRP 2}) and 
(\ref{TRP 3}) guarantee that 
$ T_i^{\alpha}:=\lim \left\langle T_i^{\beta}: \beta<\alpha  \right\rangle $ is well-defined. Preservation of conditions 
(\ref{TRP 1})-(\ref{TRP 4}) and (\ref{TRP 8}) is straightforward.   The restriction of condition  (\ref{TRP 5}) to cofinitary 
matroids and 
condition 
(\ref{TRP 6}) are kept 
by Observation 
\ref{limit obs}.
To check condition  (\ref{TRP 5}) for a finitary $ M_i $, let 
$ g\in P_i $ be arbitrary. Since 
$ \preceq_i $ is a well-order, it follows from condition (\ref{TRP 7}) that there is an $ S_g $ such that  $ S^{\beta}_g =S_g$  
for all large enough $ \beta<\alpha $. But 
then $ S_g\subseteq T_i^{\alpha} $ from which $ g\in \mathsf{span}_{M_i}(T_i^{\alpha}) $ follows. Furthermore, 
clearly $ S_g^{\alpha}=S_g $ since a finite set which is $ \prec_i $-smaller than $ S_g $ and $ M_i $-spans $ g $  would have 
appeared already before the limit. 

Suppose now that $ \alpha=\beta+1 $. If $ \bigcup_{i\in \Theta}T_i^{\beta} \supseteq E $ and the analogue of condition 
(\ref{TRP 6}) for the cofinitary $ M_i $ holds, then  $( T_i^{\beta}: i\in \Theta) $ is a desired partition of $ E $ and we 
are done. Suppose it is not the case.
If there is some  $ T_j^{\beta} $ that contains an $ M_j $-circuit $ C $ with $ C\not \subseteq R_j $, 
then we take an $ 
e\in P_j\cap C $ (see property (\ref{TRP 2}))  and define $T_j^{\beta+1}:= T_j^{\beta}-e $ and $ T_i^{\beta+1}:= T_i^{\beta} $ 
for $ i\neq j $.  The 
preservation of the conditions  (\ref{TRP 1})-(\ref{TRP 8}) is trivial. If there is no such a $ T_j^{\beta} $, then there must be 
some $ e\in E $ which is not covered by the sets $ T_i^{\beta} $. Then there is a unique $ k\in \Theta $ with $ e\in R_k $.
If $ M_k $ is cofinitary then let  $ T_k^{\beta+1}:=T_k^{\beta}+e $ and $ T_i^{\beta+1}:=T_i^{\beta} $ for $ i\neq k $.
We proceed the same way if $ M_k $ is finitary and $ T_k^{\beta}+e $ does not contain any $ M_k $-circuit $ C $ with $ C\not 
\subseteq R_k $. The preservation of the conditions is again straightforward in both cases. 

Finally assume that  $ M_k $ is finitary and 
$ T_k^{\beta}+e $ contains an $ M_k $-circuit $ C $ with $ C\subsetneq R_k $.  Let $ f $ be the $ <_k $-maximal element of 
such a $ C 
$ and  we define  $ 
T_k^{\beta+1}:= 
T_k^{\beta}+e-f$ and
$ T_i^{\beta+1}:= T_i^{\beta}$ for $ i\neq k $. Since  $ C\cap P_k\neq \varnothing $ (because $ C\not\subseteq R_k $) and  the 
elements of $ P_k $ are $ <_k 
$-larger than the elements  of $ R_k $, we  have $ f\in P_k $. Conditions (\ref{TRP 1})-(\ref{TRP 5}) remain true for obvious 
reasons. Suppose for a contradiction that condition (\ref{TRP 6}) fails and $ C' $ is an $ M_k $-circuit in $ T_k^{\beta+1} $  
with  $ C'\not \subseteq R_k $. Then
$ f\notin C' $ and we must have $ e\in C' $  since otherwise $ C'\subseteq T^{\beta}_k $  and therefore  this condition would have 
been already violated 
with respect to $ T^{\beta}_k $. By applying 
strong circuit elimination with the $M_k  $-circuits $ C $ and $ C' $, we obtain a circuit 
$ C''\subseteq C\cup C'-e $ through $ f $.  But then $ C''\subseteq T_k^{\beta} $ is an $ M_k $-circuit and $ f $ witnesses  
$ C''\not\subseteq R_k $ in violation of condition 
(\ref{TRP 6}) for $ \beta $ which is a contradiction. To check  (\ref{TRP 7}), we may assume that $ f\in S_g^{\beta} $ since 
otherwise $ S_g^{\beta}\subseteq T_k^{\beta+1} $ and thus $ S_g^{\beta+1}\preceq_k S_g^{\beta}  $. If $ 
S_g^{\beta}=\{ g \} $, then $ 
f=g $ by $ f\in S_g^{\beta} $  and by the choice of  $ f $  we have $  S^{\beta+1}_f \preceq_k C-f \prec_k \{ f \} $.
Otherwise there is an $ M_k $-circuit $ C'\ni f,g$ such that $ S^{\beta}_g=C'-g\subseteq T_k^{\beta} $. By applying strong 
circuit elimination with  $ C $ and $ C' $, we obtain a circuit 
$ C''\subseteq C\cup C'-f $ through $ g $. Since $ f\in C'\setminus C'' $ and each element of $ C''\setminus C' $ is $ \prec_k 
$-smaller than 
$ f $ (because $ f=\max_{\prec_k} C $) we may conclude that $ C''\setminus C'\prec_k C'\setminus C'' $. Thus by applying 
Observation 
\ref{obs:common part}  iteratively  we get $ C''-g\prec_k C'-g $. Therefore
\[  S^{\beta+1}_g\preceq_k C''-g\prec_k C'-g=S^{\beta}_g.\]

The recursion is done and it  terminates at some ordinal since the constructed set families $ (T_i^{\beta}: i\in \Theta) $ 
are pairwise distinct by conditions (\ref{TRP 2}), (\ref{TRP 3}) and (\ref{TRP 8}).
\end{proof}
Let us point out that the special case of Proposition \ref{prop: family CBw} in which $ P_i $ and $ 
R_i $ are bases of $ M_i $  is exactly \cite[Theorem 1.2]{erde2019base}. Now we derive Theorems \ref{thm: infmatOre} and \ref{thm: family 
CB} from Proposition \ref{prop: family CBw}:

\infmatOre*

\begin{proof}
We can assume  by contracting $ F_0\cap F_1 $ and deleting $E\setminus (F_0\cup F_1) $ in both matroids  that the sets $ F_i $ form a 
bipartition of $ E $. We apply Proposition \ref{prop: family CBw} with $ \Theta=\{ 0,1 \} $,  matroids $ M_0 $ and $ 
M_1^{*} $ and sets $ P_0:=R_1:=F_0 $ and $ 
P_1:=R_0:=F_1 $. From the resulting  bipartition $ E=T_0\sqcup T_1 $ we take $ F:=T_0 $. Then

\begin{enumerate}
\item $ \mathsf{span}_{M_0}(F)\supseteq F_0 $,
\item $ \mathsf{span}_{M_1^{*}}(E\setminus F)\supseteq F_1 $,
\item $ \mathsf{span}_{M_0^{*}}(E\setminus F)\supseteq F_0 $,
\item $ \mathsf{span}_{M_1}(F)\supseteq F_1 $.
\end{enumerate}
\end{proof}

\main*
\begin{proof}
 We may assume that $ \Theta\cap E=\varnothing $. For $ i\in  \Theta$, we construct a matroid $ M'_i $ by ``copying'' $ M_i $ to $ \{ i \}\times E 
 $ and then extending to $ \Theta\times E $ by loops.  
 For $ e\in  E$, we construct a matroid $ N'_e $ by copying $ N^{*}_i $ to $\Theta \times\{ e \} $ and then extending to $ \Theta\times E $ by 
 loops. The sets $ R'_i:=\{ i \}\times R_i $ for $ i\in \Theta $ together with the sets $R'_e:= \{i\in \Theta:\ e\notin R_i  \}\times \{ e \} $ for $ e\in E 
 $ 
 cover $ \Theta \times E $. Furthermore, the elements of the family consisting of $ P'_i:=\{ i \}\times P_i $ for $ i\in \Theta $ and $ \{i\in \Theta:\ 
 e\notin P_i  \}\times \{ e \} $ for $ e\in E $ are pairwise disjoint. Let $ \{ T'_i,\  T'_e:\ i\in \Theta, e\in E \} $ be a partition of $ \Theta \times E 
 $ obtained by applying Proposition \ref{prop: family CBw} with the matroids $  M'_i,\  N'_e $, covering $ R'_i,\  R'_e $ and packing $ P'_i, P'_e\ 
 (i\in \Theta, e \in E)$. It is easy to check that the family consisting of the projections $ T_i $ of $ T'_i $ to $ E $ for $ i\in \Theta $ is as desired.
\end{proof}
\section{Applications}\label{sec: Appl}

\subsection{Cantor-Bernstein for path-systems}

We derive Theorem \ref{Pym}  from Theorem \ref{thm: infmatOrew}.

\Pym*

\begin{proof}
 For 
$ i\in \{ 0,1 \} $, we define $ M_{i} $ to be the cycle matroid of the 
graph we obtain from $ G $ by contracting $ V_i $ 
to a single vertex. 
Then 
$E(\mathcal{P}_i)\in \mathcal{I}_{M_0}\cap \mathcal{I}_{M_1}$ for $ i\in \{ 0,1 \} $. By applying Theorem 
\ref{thm: infmatOrew} with  $ 
I_i:=E(\mathcal{P}_i)$ and $ M_{1-i} 
$, we can find an $ I\in \mathcal{I}_{M_0}\cap \mathcal{I}_{M_1}$ with $E( \mathcal{P}_{1-i})\subseteq 
\mathsf{span}_{M_i}(I) $ for $ i\in \{ 0,1 \} $. 
Then  $G[I] $ is a forest in which every tree 
meets each $ V_i $ at most once. Each connected component of $G[I] $ which meets both $ V_i $ contains a unique $ 
V_0V_1 
$-path. We define $ 
\mathcal{P} $ to be the set of these paths. It remains to show that  $ \mathcal{P} $ satisfies the requirements. Let $ v_0\in V(\mathcal{P}_i)\cap 
V_i $. It is enough to show that $ v_0 $ is reachable 
from  
$ V_{1-i} $ in $G[I] $ because then the (unique) path witnessing this is in $ \mathcal{P} $.   Consider the path
 $ P\in \mathcal{P}_i $ through $ v_0 $. Let  the vertices of  $ P $ be $ v_0,\dots, v_n $ 
enumerated in the 
path-order starting from  $ V_i $. It follows from  
$ E(P)\subseteq \mathsf{span}_{M_{1-i}}(I) $ that for every $ k< n $ either $ G[I] $ contains a path between $ v_k $ and $ v_{k+1} $ or both of 
them are reachable from $ V_{1-i} $ in $G[I] $. Vertex $ v_n $ is obviously reachable from $ V_{1-i} $ 
because it is an element of it.   If we already know that $ v_{k+1} $ is  reachable from $ V_{1-i} $ in $ G[I] $, then it follows that $ v_k $ is 
reachable as well. Thus by induction  $ v_0 $ is 
reachable 
from  
$ V_{1-i} $ in $G[I] $ which completes the proof. 
\end{proof}

\subsection{Matroid Intersection} 
\cofinal*
\begin{proof}
We start with the `cofinal' part of the statement. Let $ J\in \mathcal{I}_{M_0}\cap \mathcal{I}_{M_1} $ be given. We take 
an $ I'\in 
\mathsf{SM}(M_0, M_1) $ and  fix a 
partition $ 
E=E_0\sqcup E_1 $ such that $ I'_i:=I'\cap E_i $ spans $ E_i $ in $ M_i $ for $ i\in \{ 0,1 \} $.  By applying Theorem 
\ref{thm: infmatOre} with the 
matroids $ M_i\upharpoonright E_i $ and $ M_{1-i}.E_i $ and sets $ I'_i $ and $ J_i :=J\cap E_i$, we obtain a base $ I_i $ of 
$ 
M_i\upharpoonright E_i 
$ which is 
independent in $ M_{1-i}.E_i $ and  spans $ J_i $ in $ M_{1-i}.E_i $. We claim that $ I:=I_0\sqcup I_1 $ is as desired.
Indeed,  $ I\in  \mathsf{SM}(M_0, M_1) $ because $ I_i $ is an $ M_{1-i}.E_i $-independent 
base of $ M_i\upharpoonright E_i $.  Finally, $ I_{1-i} $ spans $ J_{1-i} $  in $ M_{i}.E_{1-i}=M_i/E_i $ and $ I_i 
\subseteq 
E_i $ spans 
$ E_i $  (which contains $ J_i $) in $ M_i $  by construction thus $ J\subseteq \mathsf{span}_{M_i}(I) $. Therefore  $ 
J\subseteq 
\mathsf{span}_{M_0}(I)\cap \mathsf{span}_{M_1}(I) $ which means $ J \trianglelefteq_{M_0, M_1} I $.

In order to show the `not necessarily upward closed' part we shall construct first a bipartite graph $ G=(V_0,V_1; E) $. We 
start with a 
double ray $ \dots, v_{-1}, v_0, v_1,\dots $ and add a new vertex $ w_i $ and new edge $ v_iw_i $ for $ i\in \{ 0,1 \} $
(see Figure \ref{fig: up close fail}). The bipartite graph 
 $ G $ induces two partition  matroids $ M_0 $ and $ M_1 $ 
on $ E $ in the way that $ I\subseteq E $ is defined to be independent in $ M_i $ if no two edges in $ I $ have a common 
end-vertex  
in $ 
V_i $. Then the elements of  $ \mathcal{I}_{M_0}\cap \mathcal{I}_{M_1} $ are exactly the matchings, moreover, matching 
$ I 
$ is in $\mathsf{SM}(M_0, M_1) $ iff one can choose exactly one vertex from each $ e\in I $ such that the resulting  set is a 
vertex cover.  Let 
 \[ I_i:=\{ v_{2k+i}v_{2k+1+i}:\  k<\omega\} \text{ for } i\in \{ 0,1 \}. \] 
On the one hand, the matchings $ I_i $ cover the same vertices 
thus   \[  I_0\trianglelefteq_{M_0, 
M_1}  I_1\trianglelefteq_{M_0, M_1} I_0. \] 
On the other hand, we claim that $ I_1 $ is strongly maximal but $ I_0  $ is not. Indeed, $ \{ v_{-2k},\ v_{2k+1}:\ k<\omega  
\} $ is a vertex 
cover (upper-left and lower-right corners on Figure \ref{fig: up close fail}) that consists of choosing exactly one end-vertex of 
each edge in $ I_1 
$  
and 
therefore witnessing $ I_1\in 
\mathsf{SM}(M_0, M_1) $.  But 
there is no such a vertex cover for $ I_0 $ because if we pick $ v_i $ from the edge $ v_0v_1 $, then we cannot choose any 
end-vertex of $ 
v_{1-i}w_{1-i} 
$. Thus $ \mathsf{SM}(M_0, M_1)  $ is not upward closed in 
$ (\mathcal{I}_{M_0}\cap \mathcal{I}_{M_1},\trianglelefteq_{M_0, M_1} )$.

\begin{figure}[h]
\centering
\begin{tikzpicture}

\node[inner sep=0pt] (v3) at (0.5,1.5) {$v_{0}$};
\node[inner sep=0pt] (v23) at (1.5,1.5) {$w_1$};
\node[inner sep=0pt] (v22) at (0.5,-0.5) {$w_0$};
\node[inner sep=0pt] (v13) at (1.5,-0.5) {$v_{1}$};

\node (v10) at (-1.5,1.5) {};
\node (v8) at (-1.5,-0.5) {};
\node (v9) at (-1,1.5) {};
\node (v6) at (-1,-0.5) {};
\node (v7) at (-0.5,1.5) {};
\node (v4) at (-0.5,-0.5) {};
\node (v5) at (0,1.5) {};
\node (v2) at (0,-0.5) {};
\node (v12) at (2,1.5) {};
\node (v15) at (2,-0.5) {};
\node (v14) at (2.5,1.5) {};
\node (v17) at (2.5,-0.5) {};
\node (v16) at (3,1.5) {};
\node (v19) at (3,-0.5) {};
\node (v18) at (3.5,1.5) {};
\node (v21) at (3.5,-0.5) {};
\node (v20) at (3.8,1.5) {};
\node (v11) at (-1.8,-0.5) {};

\draw  (v2) edge (v3);
\draw  (v4) edge (v5);
\draw  (v6) edge (v7);
\draw  (v8) edge (v9);
\draw  (v10) edge (v11);
\draw  (v12) edge (v13);
\draw  (v14) edge (v15);
\draw  (v16) edge (v17);
\draw  (v18) edge (v19);
\draw  (v20) edge (v21);

\draw[dashed]  (v3) edge (v13);
\draw[dashed]  (v2) edge (v5);
\draw[dashed]  (v4) edge (v7);
\draw[dashed]  (v6) edge (v9);
\draw[dashed]  (v8) edge (v10);
\draw[dashed]  (v12) edge (v15);
\draw[dashed]  (v14) edge (v17);
\draw[dashed]  (v16) edge (v19);
\draw[dashed]  (v18) edge (v21);

\draw[dotted]  (v22) edge (v3);
\draw[dotted]  (v13) edge (v23);
\draw  (-2,2) rectangle (4,1);
\draw  (-2,-1) rectangle (4,0);

\node at (-2.5,1.5) {$V_0$};
\node at (-2.5,-0.5) {$V_1$};
\node at (-2,0.5) {$\dots$};
\node at (4,0.5) {$\dots$};
\end{tikzpicture}
\caption{Matching $ I_0 $ consists of the dashed and $ I_1 $ consists of the normal edges.}\label{fig: up close fail}
\end{figure} 
\end{proof}

\printbibliography

\end{document}